\documentclass{amsart}
\usepackage{amsmath}
\usepackage{amssymb}
\usepackage{amsfonts}

\usepackage[hypertex,colorlinks=true,linkcolor=blue,citecolor=blue,]{hyperref}

\relax
\PassOptionsToClass{twoside}{article}
\RequirePackage{amsfonts,amssymb,amsmath,amscd,amsthm}
\RequirePackage{txfonts}
\RequirePackage{graphicx}
\RequirePackage{xcolor}
\RequirePackage{enumerate}

\begin{document}

\newcommand{\om}{\omega}
\newcommand{\si}{\sigma}
\newcommand{\la}{\lambda}
\newcommand{\ph}{\varphi}
\newcommand{\ep}{\varepsilon}
\newcommand{\tep}{\widetilde{\varepsilon}}
\newcommand{\al}{\alpha}
\newcommand{\sub}{\subseteq}
\newcommand{\hra}{\hookrightarrow}
\newcommand{\de}{\delta}

\newcommand{\RR}{{\mathbb{R}}}
\newcommand{\NN}{{\mathbb{N}}}
\newcommand{\DD}{{\mathbb{D}}}
\newcommand{\CC}{{\mathbb{C}}}
\newcommand{\ZZ}{{\mathbb{Z}}}
\newcommand{\T}{{\mathbb{T}}}
\newcommand{\TT}{{\mathbb{T}}}
\newcommand{\KK}{{\mathbb{K}}}
\newcommand{\px}{\partial X}
\newcommand{\cL}{{\mathcal{L}}}
\newcommand{\cA}{{\mathcal{A}}}
\newcommand{\cB}{{\mathcal{B}}}

\def\C{\mathbb C}
\def\R{\mathbb R}
\def\b{\mathcal B}
\def\c{\mathcal C}
\def\X{\mathbb X}
\def\U{\mathcal U}\def\M{\mathcal M}
\def\a{\mathcal A}
\def\K{\mathbb K}
\def\F{\mathbb F}
\def\b{\mathfrak{B}}
\def\f{\mathfrak{F}}
\def\x{\mathfrak{X}}
\def\Z{\mathbb Z}
\def\P{\mathbb P}
\def\v{\vartheta_\alpha}
\def\va{{\varpi}_\alpha}
\def\I{\mathbb I}
\def\H{\mathbb H}
\def\Y{\mathbb Y}
\def\E{\mathbb E}
\def\N{\mathbb N}
\def\cal{\mathcal}

\newtheorem{theorem}{Theorem}[section]
\newtheorem{lemma}{Lemma}[section]
\newtheorem{proposition}{Proposition}[section]
\theoremstyle{definition}
\newtheorem{definition}{Definition}[section]
\newtheorem{remark}{Remark}[section]
\theoremstyle{Remark}
\newtheorem{example}{Example}[section]
\theoremstyle{example}
\numberwithin{equation}{section}
\def\N{\mbox{I\hspace{-.15em}N}}
\def\R{\mbox{I\hspace{-.15em}R}}
\def\P{\mbox{I\hspace{-.15em}P}}
\def\E{\mbox{I\hspace{-.15em}E}}

\title[
\hfil Existence for a almost automorphic solution]{\bf{ Existence of almost automorphic solution in distribution for a class of stochastic integro-differential  equation driven by L\'{e}vy noise}}


\author{Mamadou Moustapha Mbaye}
\address{D\'epartement de Math\'ematiques, \,Facult\'{e} des Sciences et Technique, \,Universit\'{e} Cheikh Anta Diop,
BP 5005, \,Dakar-Fann, Senegal}
\email{mamadoumoustapha3.mbaye@ucad.edu.sn}

\author{Solym Mawaki Manou-Abi}
\address{Institut Montpelli\'erain Alexander Grothendieck, UMR CNRS 5149,\, Universit\'e de Montpellier}
\email{solym-mawaki.manou-abi@umontpellier.fr}
\address{D\'epartement Sciences et Technologies, CUFR de Mayotte, 3 Route Nationale, 97660 Dembéni}
\email{solym.manou-abi@univ-mayotte.fr}







\begin{abstract}
We  investigate  a class of  stochastic integro-differential  equations driven by L\'{e}vy noise. Under some appropriate assumptions, we establish the existence of an square-mean almost automorphic solutions in distribution. Particularly, based on Schauder's fixed point theorem, the existence of square-mean almost automorphic mild solution distribution is obtained by using the condition which is weaker than Lipschitz conditions. We provide an example to illustrate ours results.
\end{abstract}

\maketitle

\section{\textbf{Introduction}}
The aim of this work is to study the existence and uniqueness of the square-mean almost automorphic mild solutions in distibution to the following class of nonlinear stochastic integro-differential equations driven by L\'{e}vy noise in a separable Hilbert space $H$\\
\begin{eqnarray}\label{eq1.1}
     x'(t) &=& Ax(t) + g(t,x(t)) + \int_{-\infty}^{t}B_{1}(t-s)f(s,x(s))ds\\ \nonumber
     &+&  \int_{-\infty}^{t}B_{2}(t-s)h(s,x(s))dW(s)\\\nonumber
     &+&\int_{-\infty}^{t}B_{2}(t-s)\int_{|y|_{V}<1}F(s,x((s-),y)\tilde{N}(ds,dy)\\
    &+& \int_{-\infty}^{t}B_{2}(t-s)\int_{|y|_{V}\geq 1}G(s,x(s-,y)N(ds,dy) \quad  \mbox{for all}\quad t \in \mathbb{R},\nonumber
\end{eqnarray}
where $A :D(A)\subset H$ is the infinitesimal generator of a $C_{0}$-semigroup $(T(t))_{t\geq0}$, $B_{1}$ and $B_{2}$ are convolution-type kernels in $L^{1}(0,\infty)$ and $L^{2}(0,\infty)$ respectively. $g,f: \mathbb{R}\times L^{2}(P,H)\rightarrow L^{2}(P,H)$
$h: \mathbb{R}\times L^{2}(P,H)\rightarrow L(V,L^{2}(P,H))$
$F,G: \mathbb{R}\times L^{2}(P,H)\times V \rightarrow L^{2}(P,H)$; $W$ and $N$ are the L\'{e}vy-It\^{o}  decomposition components of the two-sided L\'{e}vy process $L$ (with assumptions stated in Section \ref{levy}.).

Throughout this work, we assume $(H, \|\cdot\|)$ and $(V, |\cdot|)$ are real separable Hilbert spaces. We denote by $L(V, H)$ the family of bounded linear operators from $V$ to $H$ and $L^{2}(P,H)$ is the space of all $H$-valued random variables $x$ such that\\
\begin{align*}
 \mathbb{E}\|x\|^{2}=\int_{\Omega}\|x\|^{2}dP < + \infty.
\end{align*}

The concept of almost automorphic is a natural generalization of the almost periodicity that was introduced by Bochner \cite{B}.  The basic aspects of the theory of almost automorphic functions can be found for instance in to the book \cite{N'G}.

In recent years, the study of almost periodic or almost automorphic solutions to some stochastic differential equations have been considerably investigated in lots of publications \cite{BP1,BP2,BP5,BP4,CYK1,Tog1,Tog2,Pub0,Pub3,FU1,TC} because of its significance and applications in physics, mechanics and mathematical biology. The concept of square-mean almost automorphic stochastic processes was introduced by Fu and Liu \cite{FU}. As indicated in \cite{Mel1,Mel2}, it appears that almost periodicity or automorphy in distribution sense is a more appropriate concept relatively to solutions of stochastic differential equations. Recently, the concept of Poisson square-mean almost automorphy was introduced by Liu and Sun \cite{Lvy} to deal with some stochastic evolution equations driven by L\'{e}vy noise. For the almost automorphy in distribution, its various extensions in distribution sense and the applications in stochastic differential equations, one can see \cite{Pub4,Mbaye2,Wang} for more details.

One should point out that other slightly different versions of equation (\ref{eq1.1}) have been considered in the literature. In particular, Bezandry \cite{BP2}, Xia \cite{Xia} and Mbaye \cite{Mbaye} investigated the existence and uniqueness of the solution of equation (\ref{eq1.1}) in the case when $F=G=0$ based on anotherl fixed point theorem.

The rest of this work is organized as follows. In Section 2, we make a recalling on L\'{e}vy process. In Section 3, we review some concepts and basic properties on almost automorphic and Poisson almost automorphic processes. In Section 4, by Schauder's fixed point theorem, we prove the existence of an square-mean almost automorphic mild solution in distribution of equation (\ref{eq1.1}). In Section 5, we provide an example to illustrate our results.
\section{\textbf{L\'{e}vy process}}\label{levy}
\begin{definition}
 A $V$-valued stochastic process $L = (L(t), t \geq 0)$ is called L\'{e}vy process if:
\begin{enumerate}
\item $L(0) = 0$ almost surely;
\item $L$ has independent and stationary increments;
\item $L$ is stochastically continuous, i.e. for all $\epsilon > 0$ and for all $s > 0$
\begin{equation*}
    \lim_{t\rightarrow s}P(|L(t) - L(s)|>\epsilon)=0.
\end{equation*}
\end{enumerate}
\end{definition}
Every  L\'{e}vy process is c\`{a}dl\`{a}g.\\
Let $L$ be a L\'{e}vy process.
\begin{definition}\cite{Lvy}~~
\begin{enumerate}
\item A borel $B$ in $V-\{0\}$ is bounded below if $0 \notin \overline{B}$, where $\overline{B}$ is the closure of $B$.
\item $\nu(\cdot)=\mathbb{E}(N(1,\cdot))$ is called the intensity measure associated with $L$, where $\bigtriangleup L(t)=L(t)-L(t-)$ for each $t\geq 0$ and
    \begin{equation*}
        N(t,B)(\omega):=\sharp\bigg\{0 \leq s \leq t : \bigtriangleup L(s)(\omega) \in B\bigg\} := \sum _{0 \leq s \leq t}\chi_{B}(\bigtriangleup L(s)(\omega))
    \end{equation*}
    with $L(t-)=\lim_{t\nearrow s}L(s)$ and $\chi_{B}$ being the indicator function for any Borel set $B$ in $V-\{0\}$.
\item $N(t,B)$ is called Poisson random measure if $B$ is bounded below, for each $t \geq 0$.
\item For  each $t \geq 0$ and $B$ bounded below, we define the compensated Poisson random measure by
\begin{equation*}
    \widetilde{N}(t,B)= N(t,B) - t\nu(B).
\end{equation*}
\end{enumerate}

\end{definition}

\begin{proposition}\cite{levy1,levy2}
Let $L$ be the $V$-valued L\'{e}vy process. Then there exist $a\in V$, $V$-valued Wiener process $W$ with covariance operator $Q$, and an independent Poisson random measure on $\mathbb{R}^{+}\times (V-\{0\})$ such that for each $t \geq 0$
\begin{equation*}
    L(t)= at + W(t) + \int_{|x|<1}x\widetilde{N}(t,dx) + \int_{|x|\geq1}x N(t,dx),
\end{equation*}
where the Poisson random measure $N$ has the intensity measure $\nu$ satisfying

\begin{equation}\label{intens}
    \int_{V}(|x|^{2}\wedge 1)\nu(dx)<\infty
\end{equation}
and $\widetilde{N}$ is the compensated Poisson random measure of $N$.
\end{proposition}
Let $L_{1}(t)$ and $L_{2}(t)$, $t \geq 0$ be two independent and identically distributed L\'{e}vy processes. Let
\begin{equation*}
    L(t)=\left\{\begin{array}{ll}
    L_{1}(t)\quad &\mbox{for}\quad t\geq 0,\\
    -L_{2}(-t)\quad &\mbox{for}\quad t< 0
\end{array}
\right.
\end{equation*}
\begin{remark}\label{rem}
By (\ref{intens}), it follows that $\int_{|x|\geq 1}\nu(dx)<\infty$. For convenience, we denote
\begin{equation*}
    b:=\int_{|x|\geq 1}\nu(dx).
\end{equation*}
\end{remark}
Then $L$ is a two-sided L\'{e}vy process defined on the filtered probability space $(\Omega,\mathcal{F},P,(\mathcal{F}_{t})_{t\in \mathbb{R}})$. We assume that $Q$ is a positive, self-adjoint and trace class operator on $V$, see \cite{Prato} for more details. The stochastic process $\widetilde{L}=(\widetilde{L}(t), t\in \mathbb{R})$ given by $\widetilde{L}(t):= L(t+s) - L(s)$ for some $s \in \mathbb{R}$ is also a two-sided L\'{e}vy process which shares the same law as $L$. For more details about the L\'{e}vy process, we refer to \cite{levy1,Lvy,levy2}.

\section{\textbf{Square-mean almost automorphic process}}
In this section, we recall the concepts of square-mean almost automorphic process and there basic properties.

\begin{definition}\label{defcontinuous0}
Let $x : \mathbb{R} \rightarrow L^{2}(P,H)$ be a stochastic process.
\begin{enumerate}
\item $x$  is said to be stochastically bounded if there exists $M > 0$ such that
              $$ \mathbb{E}\|x(t)\|^{2} \leq M \quad \mbox{for all}\quad t \in \mathbb{R}.$$

\item $x$ is said to be stochastically continuous if
 $$ \lim_{t \rightarrow s}\mathbb{E}\|x(t) - x(s)\|^{2} = 0 \quad \mbox{for all}\quad s \in \mathbb{R}.$$
 \end{enumerate}
 \end{definition}
 Denote by $SBC(\mathbb{R},L^{2}(P,H))$ the space of all the stochastically bounded and continuous processes. Clearly, the space $SBC(\mathbb{R},L^{2}(P,H))$ is a Banach space equipped with the following norm
    $$\|x\|_{\infty}=\sup_{t\in \mathbb{R}}(\mathbb{E}\|x(t)\|^{2})^{\frac{1}{2}}.$$

 \begin{definition}\cite{Lvy}\label{defcontinuous}
Let $J : \mathbb{R}\times V \rightarrow L^{2}(P,H)$ be a stochastic process.
\begin{enumerate}
\item $J$  is said to be Poisson stochastically bounded if there exists $M > 0$ such that
              $$\int_{V}\mathbb{E} \|J(t,x)\|^{2} \nu(dx) \leq M \quad \mbox{for all}\quad t \in \mathbb{R}.$$

\item $J$ is said to be Poisson stochastically continuous if
 $$ \lim_{t \rightarrow s}\int_{V}\mathbb{E} \|J(t,x) - J(s,x)\|^{2} \nu(dx) = 0 \quad \mbox{for all}\quad s \in \mathbb{R}.$$
 Denote by $PSBC(\mathbb{R}\times V,L^{2}(P,H))$ the space of all the stochastically bounded and continuous processes.
 \end{enumerate}
 \end{definition}

\begin{definition} \cite{FU}
Let $x : \mathbb{R} \rightarrow L^{2}(P,H)$ be a continuous stochastic process. $x$ is said be square-mean almost automorphic process if for every sequence of real numbers $(t^{'}_n)_n$ we can extract a subsequence $(t_n)_n$ such that, for some stochastic process $y : \mathbb{R} \rightarrow L^{2}(P,H)$, we have
$$\lim_{n\rightarrow +\infty}\mathbb{E}\|x(t + t_{n}) - y(t)\|^{2}=0 \quad \mbox{for all}\quad t \in \mathbb{R}$$
and
$$\lim_{n\rightarrow +\infty}\mathbb{E}\|y(t - t_{n}) - x(t)\|^{2}=0 \quad \mbox{for all}\quad t \in \mathbb{R}.$$
We denote the space off all such stochastic processes by $SAA(\mathbb{R}, L^{2}(P,H))$.
\end{definition}
\begin{theorem} \cite{FU}
$ SAA(\mathbb{R}, L^{2}(P,H))$ equipped with the norm $\|\cdot\|_{\infty}$ is a Banach space.
\end{theorem}

\begin{definition}\cite{Lvy}
Let $D : \mathbb{R}\times V \rightarrow L^{2}(P,H)$ be stochastic process. $D$ is said be Poisson square-mean almost automorphic process in $t\in \mathbb{R}$ if $D$ is Poisson continuous and for every sequence of real numbers $(t^{'}_n)_n$ we can extract a subsequence $(t_n)_n$ such that, for some stochastic process $\tilde{D} : \mathbb{R}\times V \rightarrow L^{2}(P,H)$ with $\int_{V}\mathbb{E} \|\tilde{D}(t,x)\|^{2} \nu(dx)<\infty$ such that
$$\lim_{n\rightarrow +\infty}\int_{V}\mathbb{E}\|D(t + t_{n},x) - \tilde{D}(t,x)\|^{2}\nu(dx)=0 \quad \mbox{for all}\quad t \in \mathbb{R}$$
and
$$\lim_{n\rightarrow +\infty}\int_{V}\mathbb{E}\|\tilde{D}(t - t_{n},x) - D(t,x)\|^{2}\nu(dx)=0 \quad \mbox{for all}\quad t \in \mathbb{R}.$$
We denote the space off all such stochastic processes by $PSAA(\mathbb{R}\times V ,L^{2}(P,H))$.
\end{definition}

\begin{definition}\cite{Lvy}
Let $F : \mathbb{R}\times L^{2}(P,H)\times V \rightarrow L^{2}(P,H)$ be stochastic process. $F$ is said be Poisson square-mean almost automorphic process in $t\in \mathbb{R}$ for each $Y\in L^{2}(P,H)$ if $F$ is Poisson continuous and for every sequence of real numbers $(t^{'}_n)_n$ we can extract a subsequence $(t_n)_n$ such that, for some stochastic process $\tilde{F} : \mathbb{R}\times L^{2}(P,H)\times V \rightarrow L^{2}(P,H)$ with $\int_{V}\mathbb{E} \|\tilde{F}(t,Y,x)\|^{2} \nu(dx)<\infty$ such that
$$\lim_{n\rightarrow +\infty}\int_{V}\mathbb{E}\|F(t + t_{n},Y,x) - \tilde{F}(t,Y,x)\|^{2}\nu(dx)=0 \quad \mbox{for all}\quad t \in \mathbb{R}$$
and
$$\lim_{n\rightarrow +\infty}\int_{V}\mathbb{E}\|\tilde{F}(t - t_{n},Y,x) - F(t,Y,x)\|^{2}\nu(dx)=0 \quad \mbox{for all}\quad t \in \mathbb{R}.$$
We denote the space off all such stochastic processes by $PSAA(\mathbb{R}\times L^{2}(P,H)\times V ,L^{2}(P,H))$.
\end{definition}

Let $\mathcal{P}(H)$ be the space of all Borel probability measures on $H$ with the $\beta$ metric.
\begin{equation*}
    \beta(\mu,\nu):=\sup\{|\int fd\mu - \int fd\nu|: \|f\|_{BL}\leq1\},\quad \mu,\nu \in \mathcal{P}(H),
\end{equation*}
where f are Lipschitz continuous real-valued functions on $H$ with

\begin{equation*}
    \|f\|_{BL}= \|f\|_{L} + \|f\|_{\infty}, \quad \|f\|_{L} =\sup_{x\neq y}\frac{|f(x) - f(y)|}{\|x - y\|}, \quad \|f\|_{\infty}=\sup_{x\in H}|f(x)|.
    \end{equation*}

\begin{definition}\cite{Lvy}
An $H$-valued stochastic process $Y(t)$ is said to be almost automorphic in distribution if its law $\mu(t)$ is a $\mathcal{P}(H)$-valued almost automorphic mapping, i.e. for every sequence of real numbers $(s'_n)_n$, there exist a subsequence $(s_n)_n$ and a $\mathcal{P}(H)$-valued mapping $\tilde{\mu}(t)$ such that
\begin{equation*}
    \lim_{n \rightarrow \infty}\beta(\mu(t + s_n),\tilde{\mu}(t))=0 \quad \mbox{and}\quad \lim_{n \rightarrow \infty}\beta(\tilde{\mu}(t - s_n),\mu(t))=0
\end{equation*}
hold for each $t \in \mathbb{R}$.
\end{definition}

To study the existence of mild solutions to the stochastic evolution equations (\ref{eq1.1}). we will need the following assumptions,

\begin{enumerate}

\item[(H.1)] The semigroup $T(t)$ is compact for $t > 0$ and is exponentially stable, i.e., there exists constants $K, \omega > 0$ such that
    \begin{equation}\label{semigrup}
    \|T(t)\| \leq Ke^{-\omega t}\quad \mbox{for all} \quad t\geq0.
  \end{equation}

\item[(H.2)]
The functions $f,g$ and $h$ are uniformly continuous on any bounded
subset $K$ of $L^{2}(\Omega,H)$ for each $t \in \mathbb{R}$. $F, G$ are uniformly continuous on any bounded
subset $K$ of $L^{2}(\Omega,H)$ for each $t \in \mathbb{R}$ and $x \in V$. For each bounded subset $K \subset L^{2}(\Omega,H)$, $g(\mathbb{R}, K)$, $f(\mathbb{R}, K)$, $h(\mathbb{R}, K)$ are bounded and $F(\mathbb{R}, K, V)$ and $G(\mathbb{R}, K, V)$ are Poisson stochastically bounded. Moreover
 we suppose that there exists $r > 0$ such that
 \begin{equation}\label{rrr}
 \Delta_{r}\leq \frac{\omega^{2}r}{20\theta K^{2}}
\end{equation}
 where
 \begin{align*}
 \Delta_{r}=\max\Bigg\{&\sup_{t\in \mathbb{R} \|u\|_{L^{2}}\leq r}\Big\|f(t,u)\Big\|_{L^{2}}, \sup_{t\in \mathbb{R} \|u\|_{L^{2}}\leq r}\Big\|g(t,u)\Big\|_{L^{2}}, \sup_{t\in \mathbb{R} \|u\|_{L^{2}}\leq r}\Big\|h(t,u)\Big\|_{L^{2}},\\
  &\sup_{t\in \mathbb{R} \|u\|_{L^{2}}\leq r}\int_{|y|_{V}< 1}\Big\|F(t,u,x)\Big\|_{L^{2}}\nu(dx), \sup_{t\in \mathbb{R} \|u\|_{L^{2}}\leq r}\int_{|y|_{V}\geq 1}\Big\|G(t,u,x)\Big\|_{L^{2}}\nu(dx)  \Bigg\}
\end{align*}
and $$\theta=\max \bigg(1 , \|B_{1}\|^{2}_{L^{1}(0,\infty)} , 4\|B_{2}\|^{2}_{L^{2}(0,\infty)} , 2b \|B_{2}\|^{2}_{L^{1}(0,\infty)}\bigg)$$

\item[(H.3)]
    we suppose that there exist measurable functions $m_g$, $m_f$ $m_h$, $m_F$, $m_G:$ $\mathbb{R}\longrightarrow [0,\infty)$
such that

\begin{equation}\label{Lip1}
\mathbb{E}\parallel g(t,Y)-g(t,Z) \parallel^{2}\leq m_g(t)\cdot\mathbb{E}\parallel Y - Z \parallel^{2}
\end{equation}
\begin{equation}\label{Lip2}
\mathbb{E}\parallel f(t,Y)-f(t,Z) \parallel^{2}\leq m_f(t)\cdot\mathbb{E}\parallel Y - Z \parallel^{2}
\end{equation}
\begin{equation}\label{Lip3}
\mathbb{E}\parallel (h(t,Y)-h(t,Z))Q^{\frac{1}{2}} \parallel^{2}_{L(V,L^{2}(P,H))}\leq m_h(t)\cdot \mathbb{E}\parallel Y - Z \parallel^{2}
\end{equation}
\begin{equation}\label{Lip4}
\int_{|x|_{V}<1}\mathbb{E}\|F(t,Y,x) - F(t,Z,x)\|^{2}\nu(dx) \leq m_F(t)\cdot\mathbb{E}\|Y - Z\|^{2}
\end{equation}
\begin{equation}\label{Lip5}
\int_{|x|_{V}\geq1}\mathbb{E}\|G(t,Y,x) - G(t,Z,x)\|^{2}\nu(dx) \leq m_G(t)\cdot\mathbb{E}\|Y - Z\|^{2}
\end{equation}
for all $t \in \mathbb{R}$ and  for any $Y,Z \in L^{2}(P,H)$.

\item[(H.4)] If $(u_n)_{n \in \N} \subset SBC(\mathbb{R},L^{2}(P,H))$ is uniformly bounded and uniformly convergent upon every compact subset of $\mathbb{R}$, then $g(\cdot, u_n(\cdot))$, $f(\cdot, u_n(\cdot))$, $h(\cdot, u_n(\cdot))$, and $F(\cdot, u_n(\cdot),\cdot)$, $G(\cdot, u_n(\cdot),\cdot)$ are relatively compact in $SBC(\mathbb{R},L^{2}(P,H))$, $PSBC(\mathbb{R}\times V,L^{2}(P,H))$, respectively.

\end{enumerate}

\begin{definition}
An $\mathcal{F}_{t}$-progressively measurable process $\{x(t)\}_{t\in \mathbb{R}}$ is called a mild solution on $\mathbb{R}$ of equation (\ref{eq1.1})
if it satisfies the corresponding stochastic integral equation\\
\begin{eqnarray}\label{eq1.2}
    x(t)&=&T(t-a)x(a) + \int_{a}^{t}T(t-s)g(s,x(s))ds\\ \nonumber
    &+& \int_{a}^{t}T(t-\sigma)\int_{a}^{\sigma}B_{1}(\sigma-s)f(s,x(s))dsd\sigma\\ \nonumber
    &+& \int_{a}^{t}T(t-\sigma)\int_{a}^{\sigma}B_{2}(\sigma-s)h(s,x(s))dW(s)d\sigma\\ \nonumber
     &+& \int_{a}^{t}T(t-\sigma)\int_{a}^{\sigma}B_{2}(\sigma-s) \int_{|y|_{V}<1}F(s,x(s-),y)\tilde{N}(ds,dy)d\sigma\\\nonumber
    &+& \int_{a}^{t}T(t-\sigma)\int_{a}^{\sigma}B_{2}(\sigma-s)\int_{|y|_{V}\geq 1}G(s,x(s-),y)N(ds,dy)d\sigma
    \end{eqnarray}
for all $t\geq a$.\\
\end{definition}
\begin{remark}
If we let $a\rightarrow-\infty$ in the stochastic integral equation (\ref{eq1.2}), by the exponential dissipation condition of $(T(t))_{t\geq0}$, then we obtain the stochastic process $x : \mathbb{R} \rightarrow L^{2}(\Omega,\mathbb{H})$ is a mild solution of the equation (\ref{eq1.2}) if and only if $x$ satisfies the stochastic integral equation
\begin{eqnarray}\label{eq4.3}
    x(t)&=& \int_{-\infty}^{t}T(t-s)g(s,x(s))ds + \int_{-\infty}^{t}T(t-\sigma)\int_{-\infty}^{\sigma}B_{1}(\sigma-s)f(s,x(s))dsd\sigma\\\nonumber &+&\int_{-\infty}^{t}T(t-\sigma)\int_{-\infty}^{\sigma}B_{2}(\sigma-s)h(s,x(s))dW(s)d\sigma\\\nonumber
    &+& \int_{-\infty}^{t}T(t-\sigma)\int_{-\infty}^{\sigma}B_{2}(\sigma-s) \int_{|y|_{V}<1}F(s,x(s-),y)\tilde{N}(ds,dy)d\sigma\\\nonumber
    &+& \int_{-\infty}^{t}T(t-\sigma)\int_{-\infty}^{\sigma}B_{2}(\sigma-s)\int_{|y|_{V}\geq 1}G(s,x(s-),y)N(ds,dy)d\sigma.
\end{eqnarray}
\end{remark}
\section{\textbf{Square-mean almost automorphic solutions}}
This section is devoted to the existence and the uniqueness of the square-mean almost automorphic mild solution in distribution on $\mathbb{R}$ of Eq. (\ref{eq1.1}).

Define the following integral operator,
\begin{eqnarray*}
   (\Lambda x)(t)&=& \int_{-\infty}^{t}T(t-s)g(s,x(s))ds + \int_{-\infty}^{t}T(t-\sigma)\int_{-\infty}^{\sigma}B_{1}(\sigma-s)f(s,x(s))dsd\sigma\\\nonumber &+&\int_{-\infty}^{t}T(t-\sigma)\int_{-\infty}^{\sigma}B_{2}(\sigma-s)h(s,x(s))dW(s)d\sigma\\\nonumber
    &+& \int_{-\infty}^{t}T(t-\sigma)\int_{-\infty}^{\sigma}B_{2}(\sigma-s) \int_{|y|_{V}<1}F(s,x(s-),y)\tilde{N}(ds,dy)d\sigma\\\nonumber
    &+& \int_{-\infty}^{t}T(t-\sigma)\int_{-\infty}^{\sigma}B_{2}(\sigma-s)\int_{|y|_{V}\geq 1}G(s,x(s-),y)N(ds,dy)d\sigma.
\end{eqnarray*}

We have

\begin{lemma}\label{B} If the semigroup $T(t)$ verifies (\ref{semigrup}) in assumption {\rm (H.1)} and if the functions $f,g$ and $h$ are uniformly continuous on any bounded
subset $K$ of $L^{2}(\Omega,H)$ for each $t \in \mathbb{R}$. $F, G$ are uniformly continuous on any bounded
subset $K$ of $L^{2}(\Omega,H)$ for each $t \in \mathbb{R}$ and $x \in V$. For each bounded subset $K \subset L^{2}(\Omega,H)$, $g(\mathbb{R}, K)$, $f(\mathbb{R}, K)$, $h(\mathbb{R}, K)$ are bounded and $F(\mathbb{R}, K, V)$ and $G(\mathbb{R}, K, V)$ are Poisson stochastically bounded., then the mapping $\Lambda: SBC(\mathbb{R},L^{2}(P,H)) \rightarrow SBC(\mathbb{R},L^{2}(P,H))$ is well-defined and continuous.

\end{lemma}

\begin{proof}
It is easy to see that $S$ is well-defined. To complete the proof it remains to show that $\Lambda$ is continuous. Consider an arbitrary sequence of functions $u_n \in SBC(\mathbb{R},L^{2}(P,H))$ that converges uniformly to some $u \in SBC(\mathbb{R},L^{2}(P,H))$, that is,
$\big\|u_n -u\big\|_{\infty} \to 0 \quad \mbox{as} \ \ n \to \infty.$  There exists a bounded subset $K$ of $L^{2}(\Omega,H)$ such that $u_n(t) ,u(t) \in K$ for each $t \in \mathbb{R}$ and $n=1,2,...$. By assumptions, given $\epsilon>0$, there exist $\delta>0$ and $N>0$ such that $\mathbb{E}\|u_n(t) - u(t) \|^{2}< \delta$ imply that
\begin{equation*}
\mathbb{E}\parallel g(t,u_n(t))-g(t,u(t)) \parallel^{2}<\frac{\omega^{2}\epsilon}{25K^{2}}
\end{equation*}
\begin{equation*}
\mathbb{E}\parallel f(t,u_n(t))-f(t,u(t)) \parallel^{2}<\frac{\omega^{2}\epsilon}{25K^{2}(\|B_{1}\|^{2}_{L^{1}(0,\infty)}+1)}
\end{equation*}
\begin{equation*}
\mathbb{E}\parallel h(t,u_n(t))-h(t,u(t)) Q^{\frac{1}{2}}\parallel^{2}<\frac{\omega^{2}\epsilon}{25K^{2}(\|B_{2}\|^{2}_{L^{2}(0,\infty)}+1)}
\end{equation*}
\begin{equation*}
    \int_{|x|_{V}<1}\mathbb{E}\|F(t,u_n(t),x) - F(t,u(t),x)\|^{2}\nu(dx) <\frac{\omega^{2}\epsilon}{25K^{2}(\|B_{2}\|^{2}_{L^{2}(0,\infty)}+1)}
\end{equation*}
\begin{equation*}
    \int_{|x|_{V}<1}\mathbb{E}\|G(t,u_n(t),x) - G(t,u(t),x)\|^{2}\nu(dx) <\frac{\omega^{2}\epsilon}{50K^{2}(\|B_{2}\|^{2}_{L^{2}(0,\infty)}+ b\|B_{2}\|^{2}_{L^{1}(0,\infty)}+1)}
\end{equation*}
Hence
\begin{align*}
\mathbb{E}\|(\Lambda u_n)(t) - &(\Lambda u)(t)\|^{2} =\mathbb{E}\bigg\| \int_{-\infty}^{t}T(t-s)\bigg(g(s,u_n(s))-g(s,u(s))\bigg)ds\\
&+ \int_{-\infty}^{t}T(t-\sigma)\int_{-\infty}^{\sigma}B_{1}(\sigma-s)\bigg(f(s,u_n(s))-f(s,u(s))\bigg)dsd\sigma\\\nonumber &+\int_{-\infty}^{t}T(t-\sigma)\int_{-\infty}^{\sigma}B_{2}(\sigma-s)\bigg(h(s,u_n(s))-h(s,u(s))\bigg)dW(s)d\sigma\\\nonumber
    &+\int_{-\infty}^{t}T(t-\sigma)\int_{-\infty}^{\sigma}B_{2}(\sigma-s) \int_{|y|_{V}<1}\bigg(F(s,u_n(s-),y)-F(s,u(s-),y))\bigg)
    \tilde{N}(ds,dy)d\sigma\\\nonumber
    &+ \int_{-\infty}^{t}T(t-\sigma)\int_{-\infty}^{\sigma}B_{2}(\sigma-s)\int_{|y|_{V}\geq 1}\bigg(G(s,u_n(s-),y)-G(s,u(s-),y))\bigg)N(ds,dy)d\sigma\bigg\|^{2}\\
 \end{align*}
 \begin{align*}
&\leq 5\mathbb{E}\bigg\| \int_{-\infty}^{t}T(t-s)\bigg(g(s,u_n(s))-g(s,u(s))\bigg)ds\bigg\|^{2}\\
&+ 5\mathbb{E}\bigg\| \int_{-\infty}^{t}T(t-\sigma)\int_{-\infty}^{\sigma}B_{1}(\sigma-s)\bigg(f(s,u_n(s))-f(s,u(s))\bigg)dsd\sigma\bigg\|^{2}\\
&+5\mathbb{E}\bigg\|\int_{-\infty}^{t}T(t-\sigma)\int_{-\infty}^{\sigma}B_{2}(\sigma-s)\bigg(h(s,u_n(s))-h(s,u(s))\bigg)dW(s)d\sigma\bigg\|^{2}\\
&+ 5\mathbb{E}\bigg\|\int_{-\infty}^{t}T(t-\sigma)\int_{-\infty}^{\sigma}B_{2}(\sigma-s) \int_{|y|_{V}<1}\bigg(F(s,u_n(s-),y)-F(s,u(s-),y))\bigg)
    \tilde{N}(ds,dy)d\sigma\bigg\|^{2}\\
&+ 5\mathbb{E}\bigg\|\int_{-\infty}^{t}T(t-\sigma)\int_{-\infty}^{\sigma}B_{2}(\sigma-s)\int_{|y|_{V}\geq 1}\bigg(G(s,u_n(s-),y)-G(s,u(s-),y))\bigg)N(ds,dy)d\sigma\bigg\|^{2}\\
&\leq 5(I_1 + I_2 + I_3 + I_4 + I_5).
\end{align*}
Using Cauchy-Schwartz's inequality, we get
\begin{align*}
    I_1< \frac{\epsilon}{25}\quad\quad I_2< \frac{\epsilon}{25}
\end{align*}
For $I_3$, using Cauchy-Schwartz's inequality and Ito's isometry property, we obtain
\begin{align*}
    I_3\leq \frac{K^{2}}{\omega}\int_{-\infty}^{t}e^{-\omega(t-\sigma)}\int_{-\infty}^{\sigma}\mathbb{E}\bigg\|B_{2}(\sigma-s)\bigg(h(s,u_n(s))-h(s,u(s))
    \bigg)Q^{\frac{1}{2}}
    \bigg\|^{2}dsd\sigma    < \frac{\epsilon}{25}
\end{align*}
As to $I_4$ and $I_5$, by Cauchy-Schwartz's inequality and the properties of the integral for the Poisson random measure, we have
\begin{align*}
    I_4&= \mathbb{E}\bigg\|\int_{-\infty}^{t}T(t-\sigma)\int_{-\infty}^{\sigma}B_{2}(\sigma-s) \int_{|y|_{V}<1}\bigg(F(s,u_n(s-),y)-F(s,u(s-),y))\bigg) \tilde{N}(ds,dy)d\sigma\bigg\|^{2}\\
    &\leq \frac{K^{2}}{\omega}\int_{-\infty}^{t}e^{-\omega(t-\sigma)}\mathbb{E}\bigg\|\int_{-\infty}^{\sigma}B_{2}(\sigma-s) \int_{|y|_{V}<1}\bigg(F(s,u_n(s-),y)-F(s,u(s-),y))\bigg)\tilde{N}(ds,dy)\bigg\|^{2}d\sigma\\
&\leq \frac{K^{2}}{\omega}\int_{-\infty}^{t}e^{-\omega(t-\sigma)}\int_{-\infty}^{\sigma}\|B_{2}(\sigma-s)\|^{2} \int_{|y|_{V}<1}\mathbb{E}\bigg\|F(s,u_n(s-),y)-F(s,u(s-),y))\bigg\|^{2}\nu(dy)d\sigma\\
&<\frac{\epsilon}{25}.
    \end{align*}
And
\begin{align*}
    I_5&= \mathbb{E}\bigg\|\int_{-\infty}^{t}T(t-\sigma)\int_{-\infty}^{\sigma}B_{2}(\sigma-s) \int_{|y|_{V}\geq1}\bigg(F(s,u_n(s-),y)-F(s,u(s-),y))\bigg) N(ds,dy)d\sigma\bigg\|^{2}\\
    &\leq 2\mathbb{E}\bigg\|\int_{-\infty}^{t}T(t-\sigma)\int_{-\infty}^{\sigma}B_{2}(\sigma-s) \int_{|y|_{V}\geq1}\bigg(F(s,u_n(s-),y)-F(s,u(s-),y))\bigg) \tilde{N}(ds,dy)d\sigma\bigg\|^{2}\\
    &+ 2\mathbb{E}\bigg\|\int_{-\infty}^{t}T(t-\sigma)\int_{-\infty}^{\sigma}B_{2}(\sigma-s) \int_{|y|_{V}\geq1}\bigg(F(s,u_n(s-),y)-F(s,u(s-),y))\bigg) \nu(dy)d\sigma\bigg\|^{2}\\
    &\leq \frac{2K^{2}}{\omega^{2}}\|B_{2}\|^{2}_{L^{2}(0,\infty)}\Bigg(\frac{\omega^{2}\epsilon}{50K^{2}(\|B_{2}\|^{2}_{L^{2}(0,\infty)}+ b\|B_{2}\|^{2}_{L^{1}(0,\infty)}+1)}\Bigg)\\
    &+ 2 \frac{K^{2}}{\omega}\int_{-\infty}^{t}e^{-\omega(t-\sigma)}\mathbb{E}\bigg\|\int_{-\infty}^{\sigma}B_{2}(\sigma-s) \int_{|y|_{V}\geq1}\bigg(F(s,u_n(s-),y)-F(s,u(s-),y))\bigg) \nu(dy)ds\bigg\|^{2}d\sigma\\
    &\leq \frac{2K^{2}}{\omega^{2}}\|B_{2}\|^{2}_{L^{2}(0,\infty)}\Bigg(\frac{\omega^{2}\epsilon}{50K^{2}(\|B_{2}\|^{2}_{L^{2}(0,\infty)}+ b\|B_{2}\|^{2}_{L^{1}(0,\infty)}+1)}\Bigg)
    + 2 \frac{K^{2}}{\omega}\int_{-\infty}^{t}e^{-\omega(t-\sigma)}\bigg(\int_{-\infty}^{\sigma}\bigg\|B_{2}(\sigma-s)\bigg\|ds\int_{|y|_{V}\geq1}\nu(dy)\\ &~~~~~~\cdot \int_{-\infty}^{\sigma}\bigg\|B_{2}(\sigma-s)\bigg\|\int_{|y|_{V}\geq1}\mathbb{E}\bigg\|F(s,u_n(s-),y)-F(s,u(s-),y) \bigg\|^{2}\nu(dy)ds\bigg)d\sigma\\
     &\leq \frac{2K^{2}}{\omega^{2}}\|B_{2}\|^{2}_{L^{2}(0,\infty)}\Bigg(\frac{\omega^{2}\epsilon}{50K^{2}(\|B_{2}\|^{2}_{L^{2}(0,\infty)}+ b\|B_{2}\|^{2}_{L^{1}(0,\infty)}+1)}\Bigg)\\
    &+ 2b \frac{K^{2}}{\omega}\|B_{2}\|^{2}_{L^{1}(0,\infty)}\Bigg(\frac{\omega^{2}\epsilon}{50K^{2}(\|B_{2}\|^{2}_{L^{2}(0,\infty)}+ b\|B_{2}\|^{2}_{L^{1}(0,\infty)}+1)}\Bigg)\\
    &<\frac{\epsilon}{25}.
    \end{align*}
Thus, by combining $I_{1}-I_{5}$, it follows that for each $t \in \mathbb{R}$ and $n>N$
\begin{equation*}
\mathbb{E}\|(\Lambda u_n)(t) - (\Lambda u)(t)\|^{2} <\epsilon.
\end{equation*}
 This implies that $\Lambda$ is continuous. The proof is complete.
\end{proof}

\begin{theorem} \label{Th1} Assume that assumptions $(H.1)-(H.4)$ hold and
\begin{itemize}
  \item [1-] $g,f \in SAA(\mathbb{R}\times L^{2}(P,H),L^{2}(P,H))$,
  \item [2-]$h \in SAA(\mathbb{R}\times L^{2}(P,H),L(V,L^{2}(P,H)))$,
  \item [3-]$F \in PSAA(\mathbb{R}\times L^{2}(P,H)\times V ,L^{2}(P,H))$
  \item [4-] $G \in PSAA(\mathbb{R}\times L^{2}(P,H)\times V ,L^{2}(P,H))$.
\end{itemize}
then Eq. (\ref{eq1.1}) has at least one almost automorphic in distribution mild solution on $\mathbb{R}$ provided that
\begin{equation*}
    L_g=\sup_{t\in \mathbb{R}} \int_{-\infty}^{t}e^{-\omega(t-s)}m_g(s)ds<\infty, \quad L_f=\sup_{t\in \mathbb{R}} \int_{-\infty}^{t}\|B_{1}(t-s)\|m_f(s)ds<\infty,
\end{equation*}
\begin{equation*}
 L_h=\sup_{t\in \mathbb{R}} \int_{-\infty}^{t}\|B_{2}(t-s)\|^{2}m_h(s)ds<\infty,    \quad  L_F=\sup_{t\in \mathbb{R}} \int_{-\infty}^{t}\|B_{2}(t-s)\|^{2}m_F(s)ds<\infty,
\end{equation*}

\begin{equation*}
    L_G=\sup_{t\in \mathbb{R}} \int_{-\infty}^{t}\|B_{2}(t-s)\|^{2}m_G(s)ds<\infty
\end{equation*}
and
\begin{equation}\label{Cod1}
\vartheta:= 10\frac{K^{2}}{\omega^{2}}\bigg[\omega L_g + L_f\|B_{1}\|_{L^{1}(0,\infty)} + L_h + L_F + 2\bigg(1+b\|B_{2}\|_{L^{1}(0,\infty)}\bigg)L_G\bigg]<1.\,
\end{equation}

\end{theorem}
\begin{proof}
Let $B = \{u \in SBC(\mathbb{R},L^{2}(P,H)): \|u\|^{2}_{\infty} \leq r\}$. By Lemma \ref{B} and \eqref{rrr}, it follows that $B$ is a convex and closed set satisfying $\Lambda B \subset B$.
To complete the proof, we have to prove the following statements:
\begin{enumerate}
\item[a)] That $V = \{ \Lambda u(t): u \in  B\}$ is a relatively compact subset of $L^{2}(P,H)$ for each $t \in \mathbb{R}$;
\item[b)] That $U = \{ \Lambda u: u \in B\} \subset SBC(\mathbb{R},L^{2}(P,H)$ is equi-continuous.
\item[c)] The mild solution is almost automorphic in distribution.
\end{enumerate}

To prove $a)$, fix $t \in \mathbb{R}$ and consider an arbitrary $\varepsilon > 0$. Then
\begin{eqnarray*}
   (\Lambda_{\epsilon} x)(t)&=& \int_{-\infty}^{t-\epsilon}T(t-s)g(s,x(s))ds + \int_{-\infty}^{t-\epsilon}T(t-\sigma)\int_{-\infty}^{\sigma}B_{1}(\sigma-s)f(s,x(s))dsd\sigma\\\nonumber &+&\int_{-\infty}^{t-\epsilon}T(t-\sigma)\int_{-\infty}^{\sigma}B_{2}(\sigma-s)h(s,x(s))dW(s)d\sigma\\\nonumber
    &+& \int_{-\infty}^{t-\epsilon}T(t-\sigma)\int_{-\infty}^{\sigma}B_{2}(\sigma-s) \int_{|y|_{V}<1}F(s,x(s-),y)\tilde{N}(ds,dy)d\sigma\\\nonumber
    &+& \int_{-\infty}^{t-\epsilon}T(t-\sigma)\int_{-\infty}^{\sigma}B_{2}(\sigma-s)\int_{|y|_{V}\geq 1}G(s,x(s-),y)N(ds,dy)d\sigma \\\nonumber &=& T(\epsilon)\int_{-\infty}^{t-\epsilon}T(t-\epsilon-s)g(s,x(s))ds + \int_{-\infty}^{t-\epsilon}T(t-\epsilon-\sigma)\int_{-\infty}^{\sigma}B_{1}(\sigma-s)f(s,x(s))dsd\sigma\\\nonumber &+&\int_{-\infty}^{t-\epsilon}T(t-\epsilon-\sigma)\int_{-\infty}^{\sigma}B_{2}(\sigma-s)h(s,x(s))dW(s)d\sigma\\\nonumber
    &+& \int_{-\infty}^{t-\epsilon}T(t-\epsilon-\sigma)\int_{-\infty}^{\sigma}B_{2}(\sigma-s) \int_{|y|_{V}<1}F(s,x(s-),y)\tilde{N}(ds,dy)d\sigma\\\nonumber
    &+& \int_{-\infty}^{t-\epsilon}T(t-\epsilon-\sigma)\int_{-\infty}^{\sigma}B_{2}(\sigma-s)\int_{|y|_{V}\geq 1}G(s,x(s-),y)N(ds,dy)d\sigma\\\nonumber
    &=& T(\epsilon)(\Lambda x)(t-\epsilon)
\end{eqnarray*}
and hence $V_\varepsilon := \{ (\Lambda_{\varepsilon} )x(t): x \in B\}$ is relatively compact in $L^{2}(P,H)$ as the semigroup family $T(\varepsilon)$ is compact by assumption.

Now
\begin{align*}
\mathbb{E}\|(\Lambda u)(t) - (\Lambda_{\epsilon} u)(t)\|^{2} &=\mathbb{E}\bigg\| \int_{t-\epsilon}^{t}T(t-s)g(s,u(s))ds
+ \int_{t-\epsilon}^{t}T(t-\sigma)\int_{-\infty}^{\sigma}B_{1}(\sigma-s)f(s,u(s))dsd\sigma\\\nonumber &+\int_{-\infty}^{t}T(t-\sigma)\int_{-\infty}^{\sigma}B_{2}(\sigma-s)\bigg(h(s,u_n(s))-h(s,u(s)dW(s)d\sigma\\\nonumber
    &+\int_{t-\epsilon}^{t}T(t-\sigma)\int_{-\infty}^{\sigma}B_{2}(\sigma-s) \int_{|y|_{V}<1}F(s,u(s-),y)
    \tilde{N}(ds,dy)d\sigma\\\nonumber
    &+ \int_{t-\epsilon}^{t}T(t-\sigma)\int_{-\infty}^{\sigma}B_{2}(\sigma-s)\int_{|y|_{V}\geq 1}G(s,u(s-),y)N(ds,dy)d\sigma.\bigg\|^{2}
    \end{align*}
    \begin{align*}
&\leq 5\mathbb{E}\bigg\| \int_{t-\epsilon}^{t}T(t-s)g(s,u(s))ds\bigg\|^{2}+ 5\mathbb{E}\bigg\| \int_{t-\epsilon}^{t}T(t-\sigma)\int_{-\infty}^{\sigma}B_{1}(\sigma-s)f(s,u(s))dsd\sigma\bigg\|^{2}\\
&+5\mathbb{E}\bigg\|\int_{t-\epsilon}^{t}T(t-\sigma)\int_{-\infty}^{\sigma}B_{2}(\sigma-s)h(s,u(s))dW(s)d\sigma\bigg\|^{2}\\
&+ 5\mathbb{E}\bigg\|\int_{t-\epsilon}^{t}T(t-\sigma)\int_{-\infty}^{\sigma}B_{2}(\sigma-s) \int_{|y|_{V}<1}F(s,u(s-),y)
    \tilde{N}(ds,dy)d\sigma\bigg\|^{2}\\
&+ 5\mathbb{E}\bigg\|\int_{t-\epsilon}^{t}T(t-\sigma)\int_{-\infty}^{\sigma}B_{2}(\sigma-s)\int_{|y|_{V}\geq 1}G(s,u(s-),y)N(ds,dy)d\sigma\bigg\|^{2}\\
&\leq 5\bigg[\Delta_{r}K^{2}\bigg(\int_{t-\epsilon}^{t}e^{-\omega(t-\sigma)}d\sigma \bigg)^{2}+ \Delta_{r}K^{2}\|B_{1}\|^{2}_{L^{1}(0,\infty)}\bigg(\int_{t-\epsilon}^{t}e^{-\omega(t-\sigma)}d\sigma \bigg)^{2}\\ &+\Delta_{r}K^{2}\|B_{2}\|^{2}_{L^{2}(0,\infty)}\bigg(\int_{t-\epsilon}^{t}e^{-\omega(t-\sigma)}d\sigma \bigg)^{2}
 + \Delta_{r}K^{2}\|B_{2}\|^{2}_{L^{2}(0,\infty)}\bigg(\int_{t-\epsilon}^{t}e^{-\omega(t-\sigma)}d\sigma \bigg)^{2}\\
 &+\bigg(2\Delta_{r}K^{2}\|B_{2}\|^{2}_{L^{2}(0,\infty)}+ 2b\Delta_{r}K^{2}\|B_{2}\|^{2}_{L^{1}(0,\infty)}\bigg) \bigg(\int_{t-\epsilon}^{t}e^{-\omega(t-\sigma)}d\sigma \bigg)^{2}\bigg]\\
 &\leq 5\Delta_{r}K^{2}\bigg[1+\|B_{1}\|^{2}_{L^{1}(0,\infty)}+ 4\|B_{2}\|^{2}_{L^{2}(0,\infty)}+2b\|B_{2}\|^{2}_{L^{1}(0,\infty)}\bigg]\epsilon^{2}
\end{align*}
from which it follows that $V = \{ \Lambda u(t): u \in  B\}$ is a relatively compact subset of $L^{2}(P,H)$ for each $t \in \mathbb{R}$.

We now show that $b)$ holds. Let $u\in B$ and $t_1,t_2 \in \mathbb{R}$ such that $t_1<t_2$. Similar computation as that in Lemma \ref{B}, we have
\begin{eqnarray*}
    \mathbb{E}\|(\Lambda u)(t_1) - (\Lambda u)(t_2)\|^{2} &=&\mathbb{E}\bigg\|\int_{-\infty}^{t_2}T(t_2-s)g(s,x(s))ds + \int_{-\infty}^{t_2}T(t_2-\sigma)\int_{-\infty}^{\sigma}B_{1}(\sigma-s)f(s,x(s))dsd\sigma\\\nonumber  &+&\int_{-\infty}^{t_2}T(t_2-\sigma)\int_{-\infty}^{\sigma}B_{2}(\sigma-s)h(s,x(s))dW(s)d\sigma\\\nonumber
    &+& \int_{-\infty}^{t_2}T(t_2-\sigma)\int_{-\infty}^{\sigma}B_{2}(\sigma-s) \int_{|y|_{V}<1}F(s,x(s-),y)\tilde{N}(ds,dy)d\sigma\\\nonumber
    &+& \int_{-\infty}^{t_2}T(t_2-\sigma)\int_{-\infty}^{\sigma}B_{2}(\sigma-s)\int_{|y|_{V}\geq 1}G(s,x(s-),y)N(ds,dy)d\sigma\\\nonumber
    &-&\bigg(\int_{-\infty}^{t_1}T(t_1-s)g(s,x(s))ds + \int_{-\infty}^{t_1}T(t_1-\sigma)\int_{-\infty}^{\sigma}B_{1}(\sigma-s)f(s,x(s))dsd\sigma\\\nonumber  &+&\int_{-\infty}^{t_1}T(t_1-\sigma)\int_{-\infty}^{\sigma}B_{2}(\sigma-s)h(s,x(s))dW(s)d\sigma\\\nonumber
    &+& \int_{-\infty}^{t_1}T(t_1-\sigma)\int_{-\infty}^{\sigma}B_{2}(\sigma-s) \int_{|y|_{V}<1}F(s,x(s-),y)\tilde{N}(ds,dy)d\sigma\\\nonumber
    &+& \int_{-\infty}^{t_1}T(t_1-\sigma)\int_{-\infty}^{\sigma}B_{2}(\sigma-s)\int_{|y|_{V}\geq 1}G(s,x(s-),y)N(ds,dy)d\sigma\bigg)\bigg\|^{2}\\\nonumber
                \end{eqnarray*}
    \begin{eqnarray*}
 &=&\mathbb{E}\bigg\|\bigg(T(t_2 -t_1)- I\bigg)\bigg[\int_{-\infty}^{t_1}T(t_1-s)g(s,x(s))ds + \int_{-\infty}^{t_1}T(t_1-\sigma)\int_{-\infty}^{\sigma}B_{1}(\sigma-s)f(s,x(s))dsd\sigma\\\nonumber  &+&\int_{-\infty}^{t_1}T(t_1-\sigma)\int_{-\infty}^{\sigma}B_{2}(\sigma-s)h(s,x(s))dW(s)d\sigma\\\nonumber
    &+& \int_{-\infty}^{t_1}T(t_1-\sigma)\int_{-\infty}^{\sigma}B_{2}(\sigma-s) \int_{|y|_{V}<1}F(s,x(s-),y)\tilde{N}(ds,dy)d\sigma\\\nonumber
    &+& \int_{-\infty}^{t_1}T(t_1-\sigma)\int_{-\infty}^{\sigma}B_{2}(\sigma-s)\int_{|y|_{V}\geq 1}G(s,x(s-),y)N(ds,dy)d\sigma\bigg]\\\nonumber
    &+&\int_{t_1}^{t_2}T(t_2-s)g(s,x(s))ds + \int_{t_1}^{t_2}T(t_2-\sigma)\int_{-\infty}^{\sigma}B_{1}(\sigma-s)f(s,x(s))dsd\sigma\\\nonumber  &+&\int_{t_1}^{t_2}T(t_2-\sigma)\int_{-\infty}^{\sigma}B_{2}(\sigma-s)h(s,x(s))dW(s)d\sigma\\\nonumber
    &+& \int_{t_1}^{t_2}T(t_2-\sigma)\int_{-\infty}^{\sigma}B_{2}(\sigma-s) \int_{|y|_{V}<1}F(s,x(s-),y)\tilde{N}(ds,dy)d\sigma\\\nonumber
    &+& \int_{t_1}^{t_2}T(2-\sigma)\int_{-\infty}^{\sigma}B_{2}(\sigma-s)\int_{|y|_{V}\geq 1}G(s,x(s-),y)N(ds,dy)d\sigma\bigg\|^{2}\\\nonumber
    &\leq& 6\sup_{y\in V}\mathbb{E}\bigg\|\bigg(T(t_2 -t_1)- I\bigg)y\bigg\|^{2}+6\mathbb{E}\bigg\|\int_{t_1}^{t_2}T(t_2-s)g(s,x(s))ds\bigg\|^{2}\\\nonumber
    &+& 6\mathbb{E}\bigg\|\int_{t_1}^{t_2}T(t_2-\sigma)\int_{-\infty}^{\sigma}B_{1}(\sigma-s)f(s,x(s))dsd\sigma\bigg\|^{2}\\\nonumber
    &+& 6\mathbb{E}\bigg\|\int_{t_1}^{t_2}T(t_2-\sigma)\int_{-\infty}^{\sigma}B_{2}(\sigma-s)h(s,x(s))dW(s)d\sigma\bigg\|^{2}\\\nonumber
    &+& 6\mathbb{E}\bigg\|\int_{t_1}^{t_2}T(t_2-\sigma)\int_{-\infty}^{\sigma}B_{2}(\sigma-s) \int_{|y|_{V}<1}F(s,x(s-),y)\tilde{N}(ds,dy)d\sigma\bigg\|^{2}\\\nonumber
    &+& 6\mathbb{E}\bigg\|\int_{t_1}^{t_2}T(2-\sigma)\int_{-\infty}^{\sigma}B_{2}(\sigma-s)\int_{|y|_{V}\geq 1}G(s,x(s-),y)N(ds,dy)d\sigma\bigg\|^{2}\\\nonumber
    &\leq& 6\sup_{y\in V}\mathbb{E}\bigg\|\bigg(T(t_2 -t_1)- I\bigg)y\bigg\|^{2}+ 6\bigg[\Delta_{r}K^{2}\bigg(\int_{t_1}^{t_2}e^{-\omega(t_2-\sigma)}d\sigma \bigg)^{2} \\ \nonumber
    &+& \Delta_{r}K^{2}\|B_{1}\|^{2}_{L^{1}(0,\infty)}\bigg(\int_{t-\epsilon}^{t}e^{-\omega(t-\sigma)}d\sigma \bigg)^{2}\\ \nonumber
    &+& \Delta_{r}K^{2}\|B_{2}\|^{2}_{L^{2}(0,\infty)}\bigg(\int_{t-\epsilon}^{t}e^{-\omega(t-\sigma)}d\sigma \bigg)^{2}
 + \Delta_{r}K^{2}\|B_{2}\|^{2}_{L^{2}(0,\infty)}\bigg(\int_{t-\epsilon}^{t}e^{-\omega(t-\sigma)}d\sigma \bigg)^{2}\\\nonumber
 &+& \bigg(2\Delta_{r}K^{2}\|B_{2}\|^{2}_{L^{2}(0,\infty)}+ 2b\Delta_{r}K^{2}\|B_{2}\|^{2}_{L^{1}(0,\infty)}\bigg) \bigg(\int_{t-\epsilon}^{t}e^{-\omega(t-\sigma)}d\sigma \bigg)^{2}\bigg]\\\nonumber
 &\leq& 6\sup_{y\in V}\mathbb{E}\bigg\|\bigg(T(t_2 -t_1)- I\bigg)y\bigg\|^{2}\\\nonumber
 &+& 6\Delta_{r}K^{2}\bigg[1+\|B_{1}\|^{2}_{L^{1}(0,\infty)} +4\|B_{2}\|^{2}_{L^{2}(0,\infty)}+ 2b\|B_{2}\|^{2}_{L^{1}(0,\infty)} \bigg] \bigg(\int_{t_1}^{t_2}e^{-\omega(t_2-\sigma)}d\sigma \bigg)^{2}. \\ \nonumber
\end{eqnarray*}
The right-hand side tends to $0$ independently to $u\in B$ as $t_2 \rightarrow t_1$ which implies that $U$ is right equi-continuous at $t$.  Similarly, we can show that $U$ is left equi-continuous at $t$.

Denote the closed convex hull of $\Lambda B$ by $\overline{co}\,{\Lambda B}$. Since $\overline{co}\,{\Lambda B}\subset B$ and $B$ is a closed convex, it follows that
$$\Lambda(\overline{co}\, \Lambda B) \subset \Lambda B \subset \overline{co}\ \Lambda B.$$
Further, it is not hard to see that $\overline{co}\,{\Lambda B}$ is relatively compact in $L^{2}(P,H)$. Using Arzel\`a-Ascoli theorem, we deduce that the restriction of $\overline{co}\,{\Lambda B}$ to any compact subset $I$ of $\mathbb{R}$ is relatively compact in $C(I, L^{2}(P,H))$. Thus condition (H.4) implies that $\Lambda: \overline{co}\,{\Lambda B} \mapsto \overline{co}\,{\Lambda B}$ is a compact operator. In summary, $S: \overline{co}\,{\Lambda B} \mapsto \overline{co}\,{\Lambda B}$ is continuous and compact. Using the Schauder fixed point it follows that $\Lambda$ has a fixed-point.

To end the proof, we have to check this the fixed-point is almost automorphic in distribution. Since $g$, $f$, $h$ are almost automorphic and $F$, $G$ are Poisson almost automorphic, then for every sequence of real numbers $(t^{^{\prime }}_n)_n$ we can extract a subsequence $(t_n)_n$ such that, for some stochastic processes $\widetilde{g}$, $\widetilde{f}$, $\widetilde{h}$, $\widetilde{F}$, $\widetilde{G}$
\begin{equation}\label{1}
\lim_{n\rightarrow +\infty}\mathbb{E}\|g(s+t_{n},X) - \widetilde{g}(s)\|^{2}=0, \quad\quad
\lim_{n\rightarrow +\infty}\mathbb{E}\|\widetilde{g}(s - t_{n},X) - g(s,X)\|^{2}=0;
\end{equation}
\begin{equation}\label{2}
\lim_{n\rightarrow +\infty}\mathbb{E}\|f(s+ t_{n},X) - \widetilde{f}(s,X)\|^{2}=0, \quad\quad
\lim_{n\rightarrow +\infty}\mathbb{E}\|\widetilde{f}(s - t_{n},X) - f(s,X)\|^{2}=0;
\end{equation}
\begin{equation}\label{3}
\lim_{n\rightarrow +\infty}\mathbb{E}\|\bigg(h(s+ t_{n},X) - \widetilde{h}(s,X)\bigg)Q^{\frac{1}{2}}\|^{2}_{L(V,L^{2}(P,H))}=0, \quad\quad
\lim_{n\rightarrow +\infty}\mathbb{E}\|\bigg(\widetilde{h}(s - t_{n},X) - h(s,X)\bigg)Q^{\frac{1}{2}}\|^{2}_{L(V,L^{2}(P,H))}=0;
\end{equation}
\begin{eqnarray}\label{4}
\lim_{n\rightarrow +\infty}\int_{|y|_{V}<1}\mathbb{E}\|F(s+ t_{n},X,y) - \widetilde{F}(s,X,y)\|^{2}\nu(dy)=0,\\\nonumber
\lim_{n\rightarrow +\infty}\int_{|y|_{V}<1}\mathbb{E}\|\widetilde{F}(s - t_{n},X,y) - F(s,X,y)\|^{2}\nu(dy)=0
\end{eqnarray}
and
\begin{eqnarray}\label{5}
\lim_{n\rightarrow +\infty}\int_{|y|_{V}\geq1}\mathbb{E}\|G(s+ t_{n},X,y) - \widetilde{G}(s,X,y)\|^{2}\nu(dy)=0,\\\nonumber
\lim_{n\rightarrow +\infty}\int_{|y|_{V}\geq1}\mathbb{E}\|\widetilde{G}(s - t_{n},X,y) - G(s,X,y)\|^{2}\nu(dy)=0
\end{eqnarray}
hold for each $s\in \mathbb{R}$ and $X \in L^{2}(P,H)$.

For $t \in \mathbb{R}$, we define
\begin{eqnarray*}
    \widetilde{X}(t)&=& \int_{-\infty}^{t}T(t-s)\widetilde{g}(s,\widetilde{X}(s))ds + \int_{-\infty}^{t}T(t-\sigma)\int_{-\infty}^{\sigma}B_{1}(\sigma-s)\widetilde{f}(s,\widetilde{X}(s))dsd\sigma\\\nonumber &+&\int_{-\infty}^{t}T(t-\sigma)\int_{-\infty}^{\sigma}B_{2}(\sigma-s)\widetilde{h}(s,\widetilde{X}(s))dW(s)d\sigma\\\nonumber
    &+& \int_{-\infty}^{t}T(t-\sigma)\int_{-\infty}^{\sigma}B_{2}(\sigma-s) \int_{|y|_{V}<1}\widetilde{F}(s,\widetilde{X}(s-),y)\tilde{N}(ds,dy)d\sigma\\\nonumber
    &+& \int_{-\infty}^{t}T(t-\sigma)\int_{-\infty}^{\sigma}B_{2}(\sigma-s)\int_{|y|_{V}\geq 1}\widetilde{G}(s,\widetilde{X}(s-),y)N(ds,dy)d\sigma.
\end{eqnarray*}
Let $W_{n}(s) = W(s+ t_{n}) - W(t_{n})$, $N_{n}(s,y) = N(s+ t_{n},y) - N(t_{n},x)$ and $\tilde{N}_{n}(s,y) = N(s+ t_{n},y) - N(t_{n},x)$  for each $s \in \mathbb{R}$. Then $W_{n}$ is also a Q-Wiener process having the same distribution as $W$ and $N_{n}$ have the same distribution as $N$ with compensated Poisson random measure $\tilde{N}_{n}$.

Consider the process define as follows
\begin{eqnarray*}
    X_{n}(t)&=& \int_{-\infty}^{t}T(t-s)g(s+ t_{n},X_{n}(s))ds + \int_{-\infty}^{t}T(t-\sigma)\int_{-\infty}^{\sigma}B_{1}(\sigma-s)f(s+ t_{n},X_{n}(s))dsd\sigma\\\nonumber &+&\int_{-\infty}^{t}T(t-\sigma)\int_{-\infty}^{\sigma}B_{2}(\sigma-s)h(s+ t_{n},X_{n}(s))dW(s)d\sigma\\\nonumber
    &+& \int_{-\infty}^{t}T(t-\sigma)\int_{-\infty}^{\sigma}B_{2}(\sigma-s) \int_{|y|_{V}<1}F(s+ t_{n},X_{n}(s-),y)\tilde{N}(ds,dy)d\sigma\\\nonumber
    &+& \int_{-\infty}^{t}T(t-\sigma)\int_{-\infty}^{\sigma}B_{2}(\sigma-s)\int_{|y|_{V}\geq 1}G(s+ t_{n},X_{n}(s-),y)N(ds,dy)d\sigma.
\end{eqnarray*}
Note that $X(t+ t_{n})$ and $X_{n}$ have the same law and since the convergence in $L^{2}$ implies convergence in distribution, then we have
\begin{align*}
    \mathbb{E}\|X_{n}(t) - &\widetilde{X}(t)\|^{2}\\
&\leq 5\mathbb{E}\bigg\| \int_{-\infty}^{t}T(t-s)\bigg(g(s+ t_{n},X_{n}(s))-\widetilde{g}(s,\widetilde{X}(s))\bigg)ds\bigg\|^{2}\\
&+ 5\mathbb{E}\bigg\| \int_{-\infty}^{t}T(t-\sigma)\int_{-\infty}^{\sigma}B_{1}(\sigma-s)\bigg(f(s+ t_{n},X_{n}(s))-\widetilde{f}(s,\widetilde{X}(s))\bigg)dsd\sigma\bigg\|^{2}\\
&+5\mathbb{E}\bigg\|\int_{-\infty}^{t}T(t-\sigma)\int_{-\infty}^{\sigma}B_{2}(\sigma-s)\bigg(h(s+ t_{n},X_{n}(s))-\widetilde{h}(s,\widetilde{X}(s))\bigg)dW(s)d\sigma\bigg\|^{2}\\
&+ 5\mathbb{E}\bigg\|\int_{-\infty}^{t}T(t-\sigma)\int_{-\infty}^{\sigma}B_{2}(\sigma-s) \int_{|y|_{V}<1}\bigg(F(s+ t_{n},X_{n}(s-),y)-\widetilde{F}(s,\widetilde{X}(s-),y))\bigg)
    \tilde{N}(ds,dy)d\sigma\bigg\|^{2}\\
&+ 5\mathbb{E}\bigg\|\int_{-\infty}^{t}T(t-\sigma)\int_{-\infty}^{\sigma}B_{2}(\sigma-s)\int_{|y|_{V}\geq 1}\bigg(G(s+ t_{n},X_{n}(s-),y)-\widetilde{G}(s,\widetilde{X}(s-),y))\bigg)N(ds,dy)d\sigma\bigg\|^{2}\\
&\leq 5(J_1 + J_2 + J_3 + J_4 + J_5).
\end{align*}
For $J_1$, we have
\begin{align*}
   J_1&=\mathbb{E}\bigg\| \int_{-\infty}^{t}T(t-s)\bigg(g(s+ t_{n},X_{n}(s))-\widetilde{g}(s,\widetilde{X}(s))\bigg)ds\bigg\|^{2}\\
    &\leq 2\mathbb{E}\bigg\| \int_{-\infty}^{t}T(t-s)\bigg(g(s+ t_{n},X_{n}(s))-g(s+ t_{n},\widetilde{X}(s))\bigg)ds\bigg\|^{2}\\
    &+ 2\mathbb{E}\bigg\| \int_{-\infty}^{t}T(t-s)\bigg(g(s+ t_{n},\widetilde{X}(s))-\widetilde{g}(s,\widetilde{X}(s))\bigg)ds\bigg\|^{2}\\
    &\leq \frac{2K^{2}}{\omega} \int_{-\infty}^{t}e^{-\omega(t-s)}m_{g}(s+t_{n})\mathbb{E}\bigg\|X_{n}(s)-\widetilde{X}(s)\bigg\|^{2}ds\\
    &+ \frac{2K^{2}}{\omega} \int_{-\infty}^{t}e^{-\omega(t-s)}\mathbb{E}\bigg\|g(s+ t_{n},\widetilde{X}(s))-\widetilde{g}(s,\widetilde{X}(s))\bigg\|^{2}ds\\
    &\leq \frac{2K^{2}}{\omega} \int_{-\infty}^{t}e^{-\omega(t-s)}m_{g}(s+t_{n})\mathbb{E}\bigg\|X_{n}(s)-\widetilde{X}(s)\bigg\|^{2}ds
     + \zeta_{1}^{n},
\end{align*}
where $\zeta_{1}^{n}:=\frac{2K^{2}}{\omega^{2}}\sup_{s\in \mathbb{R}}\mathbb{E}\bigg\|g(s+ t_{n},\widetilde{X}(s))-\widetilde{g}(s,\widetilde{X}(s))\bigg\|^{2}ds$. Since $\widetilde{X}(\cdot)$ is bounded in $L^{2}(P,H)$, it follows by (\ref{1}), that $\zeta_{1}^{n} \rightarrow 0$ as $n\rightarrow \infty$.

For $J_2$, we have
\begin{align*}
   J_2&=\mathbb{E}\bigg\| \int_{-\infty}^{t}T(t-\sigma)\int_{-\infty}^{\sigma}B_{1}(\sigma-s)\bigg(f(s+ t_{n},X_{n}(s))-\widetilde{f}(s,\widetilde{X}(s))\bigg)dsd\sigma\bigg\|^{2}\\
   &\leq 2\mathbb{E}\bigg\| \int_{-\infty}^{t}T(t-\sigma)\int_{-\infty}^{\sigma}B_{1}(\sigma-s)\bigg(f(s+ t_{n},X_{n}(s))-f(s+ t_{n},\widetilde{X}(s))\bigg)dsd\sigma\bigg\|^{2}\\
    &+ 2\mathbb{E}\bigg\| \int_{-\infty}^{t}T(t-\sigma) \int_{-\infty}^{\sigma}B_{1}(\sigma-s)\bigg(f(s+ t_{n},\widetilde{X}(s))-\widetilde{f}(s,\widetilde{X}(s))\bigg)dsd\sigma\bigg\|^{2}\\
    &\leq \frac{2K^{2}}{\omega} \|B_{1}\|_{L^{1}(0,\infty)} \int_{-\infty}^{t}e^{-\omega(t-\sigma)}\int_{-\infty}^{\sigma}\|B_{1}(\sigma-s)\|m_{f}(s+t_{n})\mathbb{E}\bigg\|X_{n}(s)-\widetilde{X}(s)\bigg\|^{2}dsd\sigma\\
    &+\frac{2K^{2}}{\omega} \|B_{1}\|_{L^{1}(0,\infty)} \int_{-\infty}^{t}e^{-\omega(t-\sigma)}\int_{-\infty}^{\sigma}\|B_{1}(\sigma-s)\|\mathbb{E}\|f(s+ t_{n},\widetilde{X}(s))-\widetilde{f}(s,\widetilde{X}(s))\|^{2}dsd\sigma\\
    &\leq \frac{2K^{2}}{\omega} \|B_{1}\|_{L^{1}(0,\infty)} \int_{-\infty}^{t}e^{-\omega(t-\sigma)}\int_{-\infty}^{\sigma}\|B_{1}(\sigma-s)\|m_{f}(s+t_{n})\mathbb{E}\bigg\|X_{n}(s)-\widetilde{X}(s)\bigg\|^{2}dsd\sigma
    + \zeta_{2}^{n},
\end{align*}
where $\zeta_{2}^{n}:=\frac{2K^{2}}{\omega^{2}} \|B_{1}\|^{2}_{L^{1}(0,\infty)} \sup_{s\in \mathbb{R}}\mathbb{E}\|f(s+
t_{n},\widetilde{X}(s))-\widetilde{f}(s,\widetilde{X}(s))\|^{2}$. For the same reason as for $\zeta_{1}^{n}$, $\zeta_{2}^{n} \rightarrow 0$ as $n\rightarrow \infty$.

For $J_3$, using Cauchy-Schwartz's inequality and the Ito's isometry, we have
\begin{align*}
    J_3 &= \mathbb{E}\bigg\|\int_{-\infty}^{t}T(t-\sigma)\int_{-\infty}^{\sigma}B_{2}(\sigma-s)\bigg(h(s+ t_{n},X_{n}(s))-\widetilde{h}(s,\widetilde{X}(s))\bigg)dW(s)d\sigma\bigg\|^{2}\\
    &\leq 2\mathbb{E}\bigg\| \int_{-\infty}^{t}T(t-\sigma)\int_{-\infty}^{\sigma}B_{2}(\sigma-s)\bigg(h(s+ t_{n},X_{n}(s))-h(s+ t_{n},\widetilde{X}(s))\bigg)dW(s)d\sigma\bigg\|^{2}\\
    &+ 2\mathbb{E}\bigg\| \int_{-\infty}^{t}T(t-\sigma) \int_{-\infty}^{\sigma}B_{2}(\sigma-s)\bigg(h(s+ t_{n},\widetilde{X}(s))-\widetilde{h}(s,\widetilde{X}(s))\bigg)dW(s)d\sigma\bigg\|^{2}\\
    &\leq \frac{2K^{2}}{\omega}\int_{-\infty}^{t}e^{-\omega(t-\sigma)}\int_{-\infty}^{\sigma}\|B_{2}(\sigma-s)\|^{2}m_{h}(s+t_{n})\mathbb{E}\bigg\|X_{n}(s)-
    \widetilde{X}(s)\bigg\|^{2}dsd\sigma\\
    &+ \frac{2K^{2}}{\omega}\int_{-\infty}^{t}e^{-\omega(t-\sigma)}\int_{-\infty}^{\sigma}\|B_{2}(\sigma-s)\|^{2}\mathbb{E}\bigg\|\bigg(h(s+ t_{n},\widetilde{X}(s))-\widetilde{h}(s,\widetilde{X}(s))\bigg)Q^{\frac{1}{2}}\bigg\|^{2}_{L(V,L^{2}(P,H))}dsd\sigma\\
     &\leq \frac{2K^{2}}{\omega}\int_{-\infty}^{t}e^{-\omega(t-\sigma)}\int_{-\infty}^{\sigma}\|B_{2}(\sigma-s)\|^{2}m_{h}(s+t_{n})\mathbb{E}\bigg\|X_{n}(s)-
    \widetilde{X}(s)\bigg\|^{2}dsd\sigma + \zeta_{3}^{n},
\end{align*}
where $\zeta_{3}^{n}:=\frac{2K^{2}}{\omega^{2}}\|B_{2}\|^{2}_{L^{2}(0,\infty)}
\sup_{s\in\mathbb{R}}\mathbb{E}\bigg\|\bigg(h(s+t_{n},\widetilde{X}(s))-\widetilde{h}(s,\widetilde{X}(s))\bigg)Q^{\frac{1}{2}}\bigg\|^{2}_{L(V,L^{2}(P,H))}$. By(\ref{3}), it follows that $\zeta_{3}^{n} \rightarrow 0$ as $n\rightarrow \infty$, like $\zeta_{1}^{n}$.

For $J_4$, using Cauchy-Schwartz's inequality and the properties of the integral for the Poisson random measure, we have
\begin{align*}
    J_4 &=\mathbb{E}\bigg\|\int_{-\infty}^{t}T(t-\sigma)\int_{-\infty}^{\sigma}B_{2}(\sigma-s) \int_{|y|_{V}<1}\bigg(F(s+ t_{n},X_{n}(s-),y)-\widetilde{F}(s,\widetilde{X}(s-),y))\bigg)
    \tilde{N}(ds,dy)d\sigma\bigg\|^{2}\\
    &\leq 2 \mathbb{E}\bigg\|\int_{-\infty}^{t}T(t-\sigma)\int_{-\infty}^{\sigma}B_{2}(\sigma-s)\int_{|y|_{V}<1}\bigg(F(s+ t_{n},X_{n}(s-),y)-F(s+ t_{n},\widetilde{X}(s-),y))\bigg)
    \tilde{N}(ds,dy)d\sigma\bigg\|^{2}\\
    &+ 2 \mathbb{E}\bigg\|\int_{-\infty}^{t}T(t-\sigma)\int_{-\infty}^{\sigma}B_{2}(\sigma-s) \int_{|y|_{V}<1}\bigg(F(s+ t_{n},\widetilde{X}(s-),y)-\widetilde{F}(s,\widetilde{X}(s-),y))\bigg)
    \tilde{N}(ds,dy)d\sigma\bigg\|^{2}\\
\end{align*}
\begin{align*}
    &\leq \frac{2K^{2}}{\omega}\int_{-\infty}^{t}e^{-\omega(t-\sigma)}\int_{-\infty}^{\sigma}\|B_{2}(\sigma-s)\|^{2}m_{F}(s+t_{n})\mathbb{E}\bigg\|X_{n}(s)-
    \widetilde{X}(s)\bigg\|^{2}dsd\sigma\\
    &+\frac{2K^{2}}{\omega} \int_{-\infty}^{t}e^{-\omega(t-\sigma)}\int_{-\infty}^{\sigma}\|B_{2}(\sigma-s)\|^{2} \int_{|y|_{V}<1} \mathbb{E}\bigg\|F(s+ t_{n},\widetilde{X}(s-),y)-\widetilde{F}(s,\widetilde{X}(s-),y)\bigg\|^{2}\nu(dy)dsd\sigma\\
    &\leq  \frac{2K^{2}}{\omega}\int_{-\infty}^{t}e^{-\omega(t-\sigma)}\int_{-\infty}^{\sigma}\|B_{2}(\sigma-s)\|^{2}m_{F}(s+t_{n})\mathbb{E}\bigg\|X_{n}(s)-
    \widetilde{X}(s)\bigg\|^{2}dsd\sigma + \zeta_{4}^{n},
\end{align*}
where $\zeta_{4}^{n}:=\frac{2K^{2}}{\omega^{2}}\|B_{2}\|^{2}_{L^{2}(0,\infty)}\sup_{s\in\mathbb{R}} \bigg(\int_{|y|_{V}<1} \mathbb{E}\bigg\|F(s+ t_{n},\widetilde{X}(s-),y)-\widetilde{F}(s,\widetilde{X}(s-),y)\bigg\|^{2}\nu(dy)\bigg)$. Using (\ref{4}), it follows that $\zeta_{4}^{n} \rightarrow 0$ as $n\rightarrow \infty$, like $\zeta_{1}^{n}$.

For $J_5$, using Cauchy-Schwartz's inequality and the properties of the integral for the Poisson random measure, we have
\begin{align*}
    J_5 &=\mathbb{E}\bigg\|\int_{-\infty}^{t}T(t-\sigma)\int_{-\infty}^{\sigma}B_{2}(\sigma-s)\int_{|y|_{V}\geq 1}\bigg(G(s+ t_{n},X_{n}(s-),y)-\widetilde{G}(s,\widetilde{X}(s-),y)\bigg)N(ds,dy)d\sigma\bigg\|^{2}\\
    &\leq 2\mathbb{E}\bigg\|\int_{-\infty}^{t}T(t-\sigma)\int_{-\infty}^{\sigma}B_{2}(\sigma-s)\int_{|y|_{V}\geq 1}\bigg(G(s+ t_{n},X_{n}(s-),y)-G(s+ t_{n},\widetilde{X}(s-),y)\bigg)N(ds,dy)d\sigma\bigg\|^{2}\\
    &+ 2\mathbb{E}\bigg\|\int_{-\infty}^{t}T(t-\sigma)\int_{-\infty}^{\sigma}B_{2}(\sigma-s)\int_{|y|_{V}\geq 1}\bigg(G(s+ t_{n},\widetilde{X}(s-),y)-\widetilde{G}(s,\widetilde{X}(s-),y)\bigg)N(ds,dy)d\sigma\bigg\|^{2}\\
    &\leq 4\mathbb{E}\bigg\|\int_{-\infty}^{t}T(t-\sigma)\int_{-\infty}^{\sigma}B_{2}(\sigma-s)\int_{|y|_{V}\geq 1}\bigg(G(s+ t_{n},X_{n}(s-),y)-G(s+ t_{n},\widetilde{X}(s-),y)\bigg)\tilde{N}(ds,dy)d\sigma\bigg\|^{2}\\
    &+ 4\mathbb{E}\bigg\|\int_{-\infty}^{t}T(t-\sigma)\int_{-\infty}^{\sigma}B_{2}(\sigma-s)\int_{|y|_{V}\geq 1}\bigg(G(s+ t_{n},X_{n}(s-),y)-G(s+ t_{n},\widetilde{X}(s-),y)\bigg)\nu(dy)dsd\sigma\bigg\|^{2}\\
    &+ 4\mathbb{E}\bigg\|\int_{-\infty}^{t}T(t-\sigma)\int_{-\infty}^{\sigma}B_{2}(\sigma-s)\int_{|y|_{V}\geq 1}\bigg(G(s+ t_{n},\widetilde{X}(s-),y)-\widetilde{G}(s,\widetilde{X}(s-),y)\bigg)\tilde{N}(ds,dy)d\sigma\bigg\|^{2}\\
    &+ 4\mathbb{E}\bigg\|\int_{-\infty}^{t}T(t-\sigma)\int_{-\infty}^{\sigma}B_{2}(\sigma-s)\int_{|y|_{V}\geq 1}\bigg(G(s+ t_{n},\widetilde{X}(s-),y)-\widetilde{G}(s,\widetilde{X}(s-),y)\nu(dy)dsd\sigma\bigg\|^{2}\\
    &\leq \frac{4K^{2}}{\omega}\int_{-\infty}^{t}e^{-\omega(t-\sigma)}\int_{-\infty}^{\sigma}\|B_{2}(\sigma-s)\|^{2}m_{G}(s+t_{n})\mathbb{E}\bigg\|X_{n}(s)-
    \widetilde{X}(s)\bigg\|^{2}dsd\sigma\\
    &+\frac{4K^{2}}{\omega}\int_{-\infty}^{t}e^{-\omega(t-\sigma)}\bigg(\int_{-\infty}^{\sigma}\|B_{2}(\sigma-s)\|\int_{|y|_{V}\geq1}\nu(dy)ds\\
    &~~~~\times \int_{-\infty}^{\sigma}\|B_{2}(\sigma-s)\| \int_{|y|_{V}\geq1} \mathbb{E}\bigg\|G(s+ t_{n},X_{n}(s-),y)-G(s+ t_{n},\widetilde{X}(s-),y)\bigg\|^{2}\nu(dy)ds\bigg)d\sigma\\
    &+ \frac{2K^{2}}{\omega} \int_{-\infty}^{t}e^{-\omega(t-\sigma)}\int_{-\infty}^{\sigma}\|B_{2}(\sigma-s)\|^{2} \int_{|y|_{V}\geq1} \mathbb{E}\bigg\|G(s+ t_{n},\widetilde{X}(s-),y)-\widetilde{G}(s,\widetilde{X}(s-),y)\bigg\|^{2}\nu(dy)dsd\sigma\\
    &+\frac{4K^{2}}{\omega}\int_{-\infty}^{t}e^{-\omega(t-\sigma)}\bigg(\int_{-\infty}^{\sigma}\|B_{2}(\sigma-s)\|\int_{|y|_{V}\geq1}\nu(dy)ds\\
    &~~~~\times \int_{-\infty}^{\sigma}\|B_{2}(\sigma-s)\| \int_{|y|_{V}\geq1} \mathbb{E}\bigg\|G(s+ t_{n},\widetilde{X}(s-),y)-\widetilde{G}(s ,\widetilde{X}(s-),y)\bigg\|^{2}\nu(dy)ds\bigg)d\sigma\\
            \end{align*}
    \begin{align*}
    &\leq \frac{4K^{2}}{\omega}\int_{-\infty}^{t}e^{-\omega(t-\sigma)}\int_{-\infty}^{\sigma}\|B_{2}(\sigma-s)\|^{2}m_{G}(s+t_{n})\mathbb{E}\bigg\|X_{n}(s)-
    \widetilde{X}(s)\bigg\|^{2}dsd\sigma\\
    &+\frac{4bK^{2}}{\omega}\|B_{2}\|_{L^{1}(0,\infty)}\int_{-\infty}^{t}e^{-\omega(t-\sigma)} \int_{-\infty}^{\sigma}\|B_{2}(\sigma-s)\| m_{G}(s+t_{n})\mathbb{E}\bigg\|X_{n}(s)-
    \widetilde{X}(s)\bigg\|^{2}dsd\sigma\\
    &+ \frac{2K^{2}}{\omega} \int_{-\infty}^{t}e^{-\omega(t-\sigma)}\int_{-\infty}^{\sigma}\|B_{2}(\sigma-s)\|^{2} \int_{|y|_{V}\geq1} \mathbb{E}\bigg\|G(s+ t_{n},\widetilde{X}(s-),y)-\widetilde{G}(s,\widetilde{X}(s-),y)\bigg\|^{2}\nu(dy)dsd\sigma\\
    &+\frac{4bK^{2}}{\omega}\|B_{2}\|_{L^{2}(0,\infty)}\int_{-\infty}^{t}e^{-\omega(t-\sigma)}\int_{-\infty}^{\sigma}\|B_{2}(\sigma-s)\| \int_{|y|_{V}\geq1} \mathbb{E}\bigg\|G(s+ t_{n},\widetilde{X}(s-),y)-\widetilde{G}(s,\widetilde{X}(s-),y)\bigg\|^{2}\nu(dy)dsd\sigma\\
    &\leq \frac{4K^{2}}{\omega}\int_{-\infty}^{t}e^{-\omega(t-\sigma)}\int_{-\infty}^{\sigma}\|B_{2}(\sigma-s)\|^{2}m_{G}(s+t_{n})\mathbb{E}\bigg\|X_{n}(s)-
    \widetilde{X}(s)\bigg\|^{2}dsd\sigma\\
    &+\frac{4bK^{2}}{\omega}\|B_{2}\|_{L^{1}(0,\infty)}\int_{-\infty}^{t}e^{-\omega(t-\sigma)} \int_{-\infty}^{\sigma}\|B_{2}(\sigma-s)\| m_{G}(s+t_{n})\mathbb{E}\bigg\|X_{n}(s)-
    \widetilde{X}(s)\bigg\|^{2}dsd\sigma + \zeta_{5}^{n}(t)
\end{align*}
where
\begin{align*}
    \zeta_{5}^{n}&= \frac{2K^{2}}{\omega^{2}}\|B_{2}\|^{2}_{L^{2}(0,\infty)} \sup_{s\in\mathbb{R}} \bigg(\int_{|y|_{V}\geq1} \mathbb{E}\bigg\|G(s+ t_{n},\widetilde{X}(s-),y)-\widetilde{G}(s,\widetilde{X}(s-),y)\bigg\|^{2}\nu(dy)\bigg)\\
    &+\frac{4bK^{2}}{\omega^{2}}\|B_{2}\|^{2}_{L^{1}(0,\infty)} \sup_{s\in\mathbb{R}} \bigg(\int_{|y|_{V}\geq1} \mathbb{E}\bigg\|G(s+  t_{n},\widetilde{X}(s-),y)-\widetilde{G}(s,\widetilde{X}(s-),y)\bigg\|^{2}\nu(dy)\bigg)
\end{align*}
Using  (\ref{5}), it follows that $\zeta_{5}^{n} \rightarrow 0$ as $n\rightarrow \infty$, like $\zeta_{1}^{n}$.
By combining the estimations $J_{1} - J_{5}$, we get for all $t\in \mathbb{R}$
\begin{equation*}
    \mathbb{E}\bigg\|X_{n}(t)-\widetilde{X}(t)\bigg\|^{2}\leq \zeta^{n} + \vartheta \sup_{s\in \mathbb{R}}\mathbb{E}\bigg\|X_{n}(s)-\widetilde{X}(s)\bigg\|^{2}
\end{equation*}
where $\zeta^{n}=\Sigma_{i=1}^{5}\zeta_{i}^{n}$. It follows that
\begin{align*}
   \mathbb{E}\bigg\|X_{n}(t)-\widetilde{X}(t)\bigg\|^{2}\leq \frac{\zeta^{n}}{1-\vartheta}.
\end{align*}
Since $\vartheta<1$ and $\zeta^{n}\rightarrow 0$ as $n\rightarrow \infty$ for all $t\in\mathbb{R}$ then one has
\begin{align*}
\mathbb{E}\bigg\|X_{n}(t)-\widetilde{X}(t)\bigg\|^{2}\rightarrow 0 \quad \mbox{as}\quad n\rightarrow0\quad \mbox{for each}\quad t\in\mathbb{R}.
\end{align*}
Hence we deduce that $X(t + t_n)$ converge to $\widetilde{X}(t)$ in distribution. Similarly, one can get that $\widetilde{X}(t - t_n)$ converge to $X(t)$ in distribution too. Therefore this fixed-point solution of the equation (\ref{eq1.1}) is square-mean almost automorphic in distribution. which completes the proof.

\end{proof}

\section{Example}
Consider the following stochastic integro-differential equations
\begin{equation}\label{Exp}
 \left\{\begin{array}{ll}
\frac{\partial u(t,x)}{\partial t}= \frac{\partial^{2}}{\partial x^{2}}u(t,x) + g(t,u(t,x)) + \int_{-\infty}^{t}e^{-\omega(t-s)}f(s,u(s,x))ds + \int_{-\infty}^{t}e^{-\omega(t-s)}h(s,u(s,x))dW(s)\\
\quad\quad\quad\quad\quad+ \int_{-\infty}^{t}e^{-\omega(t-s)}\theta(s,u(s,x))Z(ds),
   \quad(t,x) \in \mathbb{R}\times(0,1),\\
   u(t,0)=u(t,1)=0, ~~~~\quad t \in \mathbb{R}\\
    u(0,x)=u_0(x)=0, ~~~~\quad x \in (0,1)
\end{array}
\right.
\end{equation}
where $W$ is a $Q$-Wiener process on $L^{2}(0,1)$ with Tr $Q<\infty$ and $Z$ is a L\'{e}vy pure jump process on $L^{2}(0,1)$ which is independent of $W$ and $\omega>0$. The forcing terms are follows:
\begin{align*}
    g(t,u)=\delta\sin u(\sin t + \sin \sqrt{2}t),\quad f(t,u)=\delta\sin u(\sin t + \sin \sqrt{3}t),
\end{align*}

\begin{align*}
    h(t,u)=\delta\sin u(\sin t + \sin \sqrt{5}t), \quad  \theta(t,u)=\delta\sin u(\sin t + \sin \pi t)
\end{align*}

with $\delta>0$. Denote $H=V=L^{2}(0,1)$. In order to write the system \eqref{Exp} on the abstract form \eqref{eq1.1}, we consider the linear operator $A: D(A)\subset L^{2}(0,1)\rightarrow L^{2}(0,1)$, given by
\begin{align*}
    D(A)&=H^{2}(0,1)\cap H^{1}_{0}(0,1),\\
Ax(\xi)&=x"(\xi) \quad\mbox{for}\quad \xi \in (0,1)\quad\mbox{and}\quad x \in D(A).
\end{align*}
It is well-known that A generates a $C_0$-semigroup $(T(t))_{t\geq0}$ on $L^{2}(0,1)$ defined by
\begin{align*}
    (T(t)x)(r)=\sum_{n=1}^{\infty}e^{-n^{2}\pi^{2}t} \langle x ,e_{n}\rangle _{L^{2}}e_{n}(r),
\end{align*}
where $e_{n}(r)=\sqrt{2}\sin(n \pi r)$ for $n=1,2,....,$ and $\|T(t)\|\leq e^{-\pi^{2}t}$ for all $t\geq0$.\\
Then the system \eqref{Exp} takes the following abstract form
\begin{eqnarray*}
     u'(t) &=& Au(t) + g(t,u(t)) + \int_{-\infty}^{t}B_{1}(t-s)f(s,u(s))ds\\ \nonumber
     &+&  \int_{-\infty}^{t}B_{2}(t-s)h(s,u(s))dW(s)\\\nonumber
     &+&\int_{-\infty}^{t}B_{2}(t-s)\int_{|y|_{V}<1}F(s,u((s-),y)\tilde{N}(ds,dy)\\
    &+& \int_{-\infty}^{t}B_{2}(t-s)\int_{|y|_{V}\geq 1}G(s,u(s-,y)N(ds,dy) \quad  \mbox{for all}\quad t \in \mathbb{R},\nonumber
\end{eqnarray*}
where $u(t)=u(t,\cdot)$, $B_1(t)=B_2(t)=e^{\omega t}$ for $t\geq 0$ and
\begin{equation*}
    \theta(t,u)Z(dt)=\int_{|y|_{V}<1}F(t,u(t-),y)\tilde{N}(dt,dy) + \int_{|y|_{V}\geq 1}G(t,u(t-),y)N(dt,dy)
\end{equation*}
with
\begin{equation*}
    Z(t)=\int_{|y|_{V}<1}y\tilde{N}(t,dy) + \int_{|y|_{V}\geq 1}yN(t,dy),
\end{equation*}

\begin{equation*}
 F(t,u,y)=h(t,u)y\cdot1_{\{|y|_{V}< 1\}}\quad  \mbox{and} \quad G(t,u,y)=h(t,u)y\cdot1_{\{|y|_{V}\geq 1\}}.
\end{equation*}
Here, we assume that the L\'{e}vy pure jump process $Z$ on $L^{2}(0,1)$ is decomposed as above by the L\'{e}vy-It decomposition theorem.\\
Clearly, $f,\theta$ are almost automorphic, and $F, G$ Poisson almost automorphic and satisfying $H_{1}$--$H_{3}$. Moreover, it is easy to see that the conditions \eqref{Lip1}--\eqref{Lip5} are satisfied with
\begin{align*}
    m_{g}(t)=\delta^{2}\bigg(\sin t + \sin \sqrt{2}t\bigg)^{2}, \quad  m_{f}(t)=\delta^{2}\bigg(\sin t + \sin \sqrt{3}t\bigg)^{2}
\end{align*}
\begin{align*}
    m_{h}(t)=\delta^{2}\|Q\|_{L(V,V)}\bigg(\sin t + \sin \sqrt{5}t\bigg)^{2}, \quad  m_{F}(t)=\delta^{2}\nu(B_1(0))\bigg(\sin t + \sin \sqrt{3}t\bigg)^{2}
\end{align*}
\begin{align*}
  m_{G}(t)=\delta^{2}b\bigg(\sin t + \sin \pi t\bigg)^{2},
\end{align*}
where $B_1(0)$ is the ball in $V$ centered at the origin with radius 1.

Obviously,
\begin{equation*}
    L_g=\sup_{t\in \mathbb{R}} \int_{-\infty}^{t}e^{-\omega(t-s)}m_g(s)ds=\delta^{2}\sup_{t\in \mathbb{R}} \int_{-\infty}^{t}e^{-\omega(t-s)}m_g(s)ds\bigg(\sin s + \sin \sqrt{2}s\bigg)^{2}<\infty,
     \end{equation*}
\begin{equation*}
    L_f=\sup_{t\in \mathbb{R}} \int_{-\infty}^{t}\|B_{1}(t-s)\|m_f(s)ds=\delta^{2}\sup_{t\in \mathbb{R}} \int_{-\infty}^{t}\|B_{1}(t-s)\|\bigg(\sin s + \sin \sqrt{3}s\bigg)^{2}ds<\infty,
     \end{equation*}
     \begin{equation*}
     \quad  L_h=\sup_{t\in \mathbb{R}} \int_{-\infty}^{t}\|B_{2}(t-s)\|m_h(s)ds=\delta^{2}\|Q\|_{L(V,V)}\sup_{t\in \mathbb{R}} \int_{-\infty}^{t}\|B_{2}(t-s)\bigg(\sin s + \sin \sqrt{5}s\bigg)^{2}ds<\infty,
\end{equation*}
\begin{equation*}
 L_F=\sup_{t\in \mathbb{R}} \int_{-\infty}^{t}\|B_{2}(t-s)\|^{2}m_F(s)ds=\delta^{2}\nu(B_1(0))\sup_{t\in \mathbb{R}} \int_{-\infty}^{t}\|B_{2}(t-s)\|^{2}\bigg(\sin s + \sin \pi s\bigg)^{2}ds<\infty,
\end{equation*}
\begin{equation*}
    L_G=\sup_{t\in \mathbb{R}} \int_{-\infty}^{t}\|B_{2}(t-s)\|^{2}m_G(s)ds=\delta^{2}b\sup_{t\in \mathbb{R}} \int_{-\infty}^{t}\|B_{2}(t-s)\|^{2}\bigg(\sin s + \sin \pi s\bigg)^{2}ds<\infty
\end{equation*}

Therefore, by Theorem \ref{Th1}, the equation \eqref{Exp} has a square-mean almost automorphic in distribution mild solution on $\mathbb{R}$ whenever $\delta$ is small enough.

\end{document}